\documentclass{amsart}
\usepackage{hyperref,longtable}
\usepackage[table]{xcolor}
\usepackage{mathabx}
\usepackage{amsthm,amsmath,amssymb,tikz}
\usepackage[outer=1.25in, inner=1.25in, top=1.25in, bottom=1.25in]{geometry}
\usepackage[shortlabels]{enumitem}
\numberwithin{equation}{section}
\usepackage[english]{babel}

\newcommand{\C}{\mathbf{C}}
\newcommand{\F}{\mathcal{F}}
\newcommand{\f}{\mathbf{f}}
\newcommand{\supp}{\mathrm{supp}}

\newcommand{\bL}{{\mathbf{L}}}
\newcommand{\bK}{{\mathbf{K}}}

\renewcommand{\wr}{\mathop{\mathrm{wr}}}
\newcommand{\mG}{{\mathcal{G}}}
\newcommand{\mJ}{{\mathcal{J}}}
\newcommand{\vG}{\Gamma}
\newcommand{\MP}{{\mathcal{P}}}

\newcommand{\Alt}{{\mathrm{Alt}}}
\newcommand{\Sym}{{\mathrm{Sym}}}
\newcommand{\SOplus}{{\mathrm{SO}^+_{2m}}}
\newcommand{\SOminus}{{\mathrm{SO}^-_{2m}}}

\newcommand{\Fix}{{\mathrm{Fix}}}
\newcommand{\fix}{{\mathrm{fix}}}
\newcommand{\RelFix}{{\mathrm{RelFix}}}
\newcommand{\Aut}{{\mathrm{Aut}}}
\newcommand{\V}{{\mathrm{V}}}
\newcommand{\id}{{\mathrm{id}}}
\newcommand{\val}{\mathrm{val}}

\subjclass[2010]{05C25, 20B25}
\keywords{Vertex-primitive, fixity, product action, digraph, graph}

\begin{document}
	
	\newtheorem{thm}{Theorem}[]
	\newtheorem{lem}[thm]{Lemma}
	\newtheorem{cor}[thm]{Corollary}
	\theoremstyle{definition}
	\newtheorem{de}[thm]{Definition}
	\newtheorem{remark}[thm]{Remark}
	\newtheorem{con}[thm]{Construction}
	\newtheorem{question}[thm]{Question}

	\title{Vertex-primitive digraphs with large fixity}
	
	\author[M.~Barbieri]{Marco Barbieri}
	\address{Dipartimento di Matematica ``Felice Casorati", University of Pavia, Via Ferrata 5, 27100 Pavia, Italy} 
	\email{marco.barbieri07@universitadipavia.it}
	
	\author[P.~Poto\v{c}nik]{Primo\v{z} Poto\v{c}nik}
	\address{%Primo\v{z} Poto\v{c}nik,\newline
	Faculty of Mathematics and Physics,
	University of Ljubljana, Slovenia; \newline
	also affiliated with\newline
	Institute of Mathematics, Physics, and Mechanics,
	Ljubljana, Slovenia}
	\email{primoz.potocnik@fmf.uni-lj.si}

	\begin{abstract}
		The relative fixity of a digraph $\Gamma$ is defined as the ratio between the largest number of vertices fixed by a nontrivial automorphism of $\Gamma$ and the number of vertices of $\Gamma$.
		We characterize the vertex-primitive digraphs whose relative fixity is at least $\frac{1}{3}$, and we show that there are only finitely many vertex-primitive graphs of bounded out-valency and relative fixity exceeding a positive constant.
	\end{abstract}

	\maketitle

\section{Introduction}
\label{sec:intro}

Throughout this paper, we use the word \emph{digraph} to denote a combinatorial structure $\vG$ determined by a finite nonempty set of {\em vertices} $V\vG$ and a set of {\em arcs}  $A\vG \subseteq V\vG\times V\vG$, sometimes also viewed as a binary relation on $V\vG$.
If the set $A\vG$ is symmetric (when viewed as a binary relation on $V\vG$),
then the digraph $\vG$ is called a {\em graph} and unordered pairs $\{u,v\}$ such that 
$(u,v)$
and $(v,u)$ are arcs
%$\in A\vG$
are called {\em edges} of $\vG$.

The {\em fixity} of a finite digraph $\vG$, denoted by $\Fix(\vG)$, is defined as the largest number of vertices that are left fixed by a nontrivial automorphism of $\vG$, while the {\em relative fixity of $\vG$} is defined as the ratio
\[\RelFix(\vG) = \frac{\Fix(\vG)} {|V\vG|} \,.\]

The notion of fixity of (di)graphs was introduced in a 2014 paper of L.~Babai~\cite{Babai1} (see also \cite{Babai2}),
where several deep results regarding the fixity of strongly regular graphs were proved (these results were later used in his work
on the graph isomorphism problem \cite{Babai3}). To convey the flavour of his work, let us mention \cite[Theorem 1.6]{Babai2},
which states that the relative fixity of a strongly regular graph (other then a complete bipartite graph or the line graph of a complete graph)
is at most $\frac{7}{8}$.

The study of the fixity of graphs continued in a series of papers \cite{BarbieriGrazianSpiga,LehnerPotocnikSpiga,PotocnikSpiga} 
by P.~Spiga and coauthors (including the authors of the present paper),
where the problem was studied in the context of vertex-transitive graphs of fixed valency.

Let us mention that fixity is a well studied parameter in the slightly more general context of permutation groups, where, instead of fixity, it is more common to consider the dual notion of {\em minimal degree} of a permutation group $G$, defined by
\[ \mu(G) = \min_{g\in G \setminus \{1_G\}} |\supp(g)| \,,\]
where $\supp(g)$ denotes the set of all non-fixed points of $g\in G$.
Note that the fixity of a digraph $\vG$ and the minimal degree of its automorphism group $\Aut(\vG)$ 
are related via the equality
\[\Fix(\vG) = |V(\vG)| - \mu(\Aut(\vG)) \,.\]

A vast majority of papers
on the topic of minimal degree of permutation groups
(including the original work of Jordan on primitive permutation groups of minimal degree $c$ for a fixed constant $c$)
concentrates on
{\em primitive permutation groups} (see, for example, \cite{Babai4,BurnessGuralnick,GuralnickMagaard,Liebeck,LiebeckSaxl,LiebeckShalev}).
It is thus natural to ask the following question:

\begin{question} 
	What can be said about a digraph with large relative fixity whose automorphism group acts primitively on the vertex-set?
\end{question}

In this paper, we answer this question in the setting where the relative fixity is more than $\frac{1}{3}$.
In our analysis, we rely heavily on the recent classification
of primitive permutation groups of minimal degree at most $\frac{2}{3}$ of the degree of the permutation group from \cite{BurnessGuralnick}.
The essence of our work thus consists of determining the digraphs upon which the permutation groups from this classification act upon.

Before stating our main result, let us first introduce a few graph theoretical concepts and constructions.
First, recall that the {\em direct product of the family of digraphs} $\vG_1, \ldots, \vG_r$ (sometimes also called the {\em tensor product} or the {\em categorical product})
is the digraph $\vG_1\times \ldots \times \vG_r$ whose vertex-set is the cartesian product  $V\vG_1\times \ldots \times V\vG_r$ and whose arc-set is
\[A(\vG_1\times \ldots \times \vG_r) = \left\{ \bigl( (u_1,\ldots,u_r ),\, (v_1,\ldots,v_r)\bigr) \,\middle\vert\, (u_i,v_i)\in A\vG_{i} \hbox{ for all } i \in \{1,\ldots,r\} \right\} \,.\]
Recall also that a {\em union of digraphs} $\vG_1$ and $\vG_2$ is the digraph whose vertex-set and arc-set are the sets $V\vG_1 \cup V\vG_2$ and
$A\vG_1 \cup A\vG_2$, respectively. Note that when $\vG_1$ and $\vG_2$ share the same vertex-set, their union is then obtained simply by taking the union of their arc-sets. Further, for a positive integer $m$, let $\bL_m$ and $\bK_m$ denote the {\em loop graph} and the {\em complete graph} on a vertex-set 
$V$ of cardinality $m$ and with arc-sets $\{(v,v) : v\in V\}$ and $\{(u,v) :  u,v\in V, u\not=v\}$, respectively.

We now have all the ingredients needed to present a construction yielding the digraph appearing in our main result.

\begin{con}
\label{con:general}
 	Let $\mG=\{\vG_0,\vG_1, \ldots, \vG_k\}$ be a list of $k+1$ pairwise distinct digraphs sharing the same vertex-set $\Delta$.
 	Without loss of generality, we shall always assume that $\vG_0 = \bL_m$ with $m = |\Delta|$. Further, let $r$ be a positive integer, and let $\mJ$ be a subset of the $r$-fold cartesian power $X^r$, where
 	$X=\{0,1,\ldots, k\}$.
 	Given this input, construct the digraph
 	\[\MP(r,\mG,\mJ) = \bigcup_{(j_1,j_2,\ldots,j_r) \in \mJ} \vG_{j_1} \times \vG_{j_2} \times \ldots \times \vG_{j_r}\]
 	and call it the {\em merged product action digraph}.
\end{con}

\begin{remark}
\label{rem:ex}
	We give some example to give a flavour of what can be obtained using Construction~\ref{con:general}.
	
	If $r=1$, then $\MP(1,\mG,\mJ)$ is simply the union of some digraphs from the set $\mG$.
	
	If $r=2$ and $\mJ=\{(1,0),(0,1)\}$, then $\MP(1,\mG,\mJ)=\bL_m \times \Gamma_1 \cup \Gamma_1 \times \bL_m$, which is, in fact, the {\em Cartesian product}	$\Gamma \square \Gamma$. (This product is sometimes called the {\em box product}, and we refer to \cite{ImrichKlavzar} for the definition of the Cartesian product.)
	
	More generally, 
	if $\mJ = \{ e_i \mid i \in \{1,\ldots,r\}\}$, where $e_i=(0,\ldots,0,1,0,\ldots,0)$ is the $r$-tuple with $1$ in the $i$-th component and zeroes elsewhere, 
	then $\MP(r,\mG,\mJ) = (\Gamma_1)^{\square r}$, the $r$-th Cartesian power of the graph $\Gamma_1\in \mG$. More specifically,
	if $\Gamma_1 = \bK_m$ and $\mJ$ is as above, then $\MP(r,\mG,\mJ)$ is the {\em Hamming graph} $\mathbf{H}(r,m) = \bK_m^{\square r}$.
\end{remark}
	
While $\mJ$ can be an arbitrary set of $r$-tuples in $X^r$, 
we will be mostly interested in the case where $\mJ \subseteq X^r$ is invariant under the induced action of some permutation group $H\le \Sym(r)$ on the set $X^r$ 
given by the rule
\[(j_1, j_2, \ldots, j_r)^h = (j_{1 h^{-1}},j_{2 h^{-1}}, \ldots, j_{r h^{-1}}) \,.\]
(Throughout this paper, in the indices, we choose to write $ih^{-1}$ instead of $i^{h^{-1}}$ for improved legibility.)
We shall say that $\mJ$ is an {\em $H$-invariant subset of $X^r$} in this case.
A subset $\mJ\subseteq X^r$ which is $H$-invariant for some {\em transitive} subgroup of $\Sym(r)$ will be called {\em homogeneous}.

The last example of Remark~\ref{rem:ex} justifies the introduction of the following new family of graphs.

\begin{de}
\label{de:genHam}
	Let $r,m$ be two positive integers, and let $\mJ\subseteq \{0,1\}^r$ be a homogeneous set. The graph $\MP \left( r, \{\bL_m,\bK_m\}, \mJ \right)$ is called \emph{generalised Hamming graph} and is denoted by $\mathbf{H}(r,m,\mJ)$.
\end{de}

\begin{remark}
\label{rem:orbHam}
	The generalised Hamming graphs $\mathbf{H}(r,m,\mJ)$, where $\mJ$ is $H$-invariant, are precisely the unions of orbital graphs for the group $\Sym(m) \wr H$ endowed with the product action (see Lemma~\ref{lem:Ham} for further details).
\end{remark}

Furthermore, a homogeneous set $\mJ$ is said to be \emph{Hamming} if,
\[\mJ = \bigcup\limits_{h\in H} \left( (X\setminus\{0\})^a\times X^b \times \{0\}^{r-a-b} \right)^h \,,\]
for some nonnegative integers $a,b$ such that $a+b\le r$ and a transitive group $H\le \Sym(r)$. It is said to be \emph{non-Hamming} otherwise.

\begin{remark}
\label{rem:isomHam}
	Let $\MP(r,\mG,\mJ)$ be a merged product action digraph, where the digraphs in $\mG$ have $m$ vertices, and where $\mJ$ is a Hamming set. Build $\mJ'\subseteq \{0,1\}^r$ from $\mJ$ by substituting any nonzero entry of a sequence in $\mJ$ with $1$. Then
	\[\MP \left( r,\mG,\mJ \right) = \MP \left( r, \{\bL_m,\bK_m\},\mJ' \right) \,.\]
	In particular, a generalised Hamming graph arises from Construction~\ref{con:general} if and only if $\mJ$ is a Hamming set.
\end{remark}

\begin{remark}
	The ordering of the Cartesian components in the definition of a Hamming set does not matter: indeed, a permutation of the components corresponds to a conjugation of the group $H$ in $\Sym(r)$, thus defining isomorphic digraphs in Construction~\ref{con:general}.
\end{remark}

We are ready to state our main result.

\begin{thm}
	\label{thm:main}
	
	Let $\vG$ be a finite vertex-primitive digraph with at least one arc. Then
	\[\RelFix(\vG) > \frac{1}{3}\] 
	if and only if one of the following occurs:
	\begin{enumerate}[$(i)$]
		
		\item $\vG$ is a generalised Hamming graph $\mathbf{H}(r,m,\mJ)$, with $m\ge 4$, and
		\[\RelFix(\vG) = 1 - \frac{2}{m} \,;\]
		
		\item $\vG$ is a merged product action graph $\MP(r,\mG,\mJ)$, where $r\ge 1$, where $\mJ$ is a non-Hamming subset of $X^r$ with $X = \{0,1,\ldots, |\mG|-1\}$, and where $\mG$ is as in one of the following:		
		\begin{enumerate}[$(a)$]
		
			\item $\mG = \{\mathbf{J}(m,k,i) \mid i \in \{0,1,\ldots,k\}\}$ is the family of distance-$i$ Johnson graphs, where $k, m$ are fixed integers such that $k\ge 2$ and $m \ge 2k+2$ (see Section~\ref{sec:mergedJordan} for details), and
			\[\RelFix(\vG) = 1 - \frac{2k(m-k)}{m(m-1)} \,;\]
			
			\item $\mG = \{\mathbf{QJ}(2m,m,i) \mid i \in \{0,1,\ldots,\lfloor m/2 \rfloor\}\}$ is the family of squashed distance-$i$ Johnson graphs,
			where $m$ is a fixed integer with $m\ge 4$ (see Section~\ref{sec:quotJordan} for details), and \[\RelFix(\vG) = \frac{1}{2}\left( 1 - \frac{1}{2m-1} \right) \,;\]
			
			\item $\mG = \{\bL_m,\vG_1,\vG_2\}$, where $\vG_1$ is a strongly regular graph listed in Section~\ref{sec:SRG}, $\vG_2$ is its complement, and \[\RelFix(\vG)=\RelFix(\vG_1)\]
			(the relative fixities are collected in Table~\ref{table}).
		\end{enumerate}
	\end{enumerate}
\end{thm}

\begin{remark}
\label{rem:graph}
	Although we do not assume that a vertex-primitive digraph $\vG$ in Theorem~\ref{thm:main} is a graph,
	the assumption of large relative fixity forces it to be such. In other words, every vertex-primtive digraph of relative fixity larger than $\frac{1}{3}$ is a graph.
\end{remark}
%This follows from the classical result by G.~Highman that states that all the nondiagonal orbital digraphs for a primitive group are connected (see, for instance, \cite[Theorem~3.2A]{DM}).

\begin{remark}
	The relative fixity can be arbitrarily close to $1$. Indeed, this can be achieved by choosing a generalised Hamming graph $\mathbf{H}(r,m,\mJ)$ with $m$ arbitrarily large.
\end{remark}

By analysing the vertex-primitive graphs of relative fixity more than $\frac{1}{3}$, one can notice that the out-valency of these graphs must grow as the number of vertices grows.
 More explicitly, a careful inspection of the families in Theorem~\ref{thm:main} leads to the following result, the proof of which we leave out.

\begin{remark}
	\label{cor:growth}
	There exists a  constant $C$ such that every
	finite connected vertex-primitive digraph $\vG$ with
	\[\RelFix(\vG) > \frac{1}{3}\]
	satisfies
	\[\val(\vG)\ge C \log\left( | V\vG | \right) \,.\]
\end{remark}

Observe that, for the Hamming graphs $\mathbf{H}(r,m)$
with $m\ge 4$, we have that
\[\val\left(\mathbf{H}(r,m)\right) = r(m-1) \ge r \log(m) = \log\left(|V\mathbf{H}(r,m)|\right) \,.\]
In particular, as both expressions are linear in $r$, a logarithmic bound in Remark~\ref{cor:growth} is the best that can be achieved.

One of the consequences of Remark~\ref{cor:growth} is that for every positive integer $d$ 
there exist only finitely many connected vertex-primitive digraphs of out-valency at most $d$ and relative fixity exceeding $\frac{1}{3}$.

As Theorem~\ref{thm:mainX} and Corollary~\ref{cor:graphs1} show,
this remains to be true if $\frac{1}{3}$ is substituted by an arbitrary positive constant.
 We thank P.~Spiga for providing us with the main ideas used in the proof.

%During the preparation of this manuscript we posed this problem to P.~Spiga, who answered it in the positive. With his permission and blessing, we are including his proof in Section~\ref{sec:X}.

\begin{thm}
	\label{thm:mainX}
	Let $\alpha$ and $\beta$ be two positive constants,
	and let $\F$ be a family of quasiprimitive permutation groups $G$ on $\Omega$ satisfying:
	\begin{enumerate}[$(a)$]
	  \item $\mu(G)\le (1-\alpha)|\Omega|$; and
	  \item $|G_\omega|\le\beta$ for every $\omega \in \Omega$.
	\end{enumerate}
	Then $\F$ is a finite family. 
\end{thm}

\begin{cor}
	\label{cor:graphs1}
	Let $\alpha$ be a positive constant, and let $d$ be a positive integer. 
	There are only finitely many vertex-primitive digraphs of out-valency at most $d$ and relative fixity exceeding $\alpha$.
\end{cor}

The proof of Theorem~\ref{thm:main} can be found in Section~\ref{sec:proof}, while Theorem~\ref{thm:mainX} and Corollary~\ref{cor:graphs1} are proved in Section~\ref{sec:X}.

\section{Basic concepts and notations}
\label{sec:notation}

\subsection{Product action}
We start by recalling the definition of a wreath product and its product action. By doing so, we also settle the notation for the rest of the paper. We refer to \cite[Section~2.6 and~2.7]{DM} for further details.

Let $H$ be a permutation group on a finite set $\Omega$. Suppose that $r=|\Omega|$, and, without loss of generality, identify $\Omega$ with the set $\{1,2,\ldots, r\}$. For an arbitrary set $X$, we may define a {\em permutation action of $H$ of rank $r$ over $X$} as the the action of $H$ on the set $X^r$ given by the rule
\[(x_1,x_2,\ldots,x_r)^h = \left( x_{1h^{-1}}, x_{2h^{-1}}, \ldots, x_{rh^{-1}} \right) \,.\]

Let $K$ be a permutation group on a set $\Delta$. We can consider the permutation action of $H$ of rank $r$ over $K$ by letting
\[(k_1,k_2,\ldots,k_r)^h = (k_{1h^{-1}},k_{2h^{-1}}, \ldots, k_{rh^{-1}}) \quad \hbox{for all } (k_1,k_2,\ldots,k_r)\in K^r, \, h \in H \,.\]
If we denote by $\vartheta$ the homomorphism $H \to \Aut(K^r)$ corresponding to this action, then the \emph{wreath product of $K$ by $H$}, in symbols $K\wr H$, is the semidirect product $K^r\rtimes_\vartheta H$. We call $K^r$ the \emph{base group}, and $H$ the \emph{top group} of this wreath product.

Note that the base and the top group are both embedded into $K\wr H$
via the monomorphisms
\[(k_1,k_2,\ldots,k_r)\mapsto \left((k_1,k_2,\ldots,k_r),1_H\right)\]
and
\[ h\mapsto \left((1_K,1_K,\ldots,1_K),h\right) \,.\]
In this way, we may view the base and the top group as subgroups of the wreath product and identify an element $((k_1,k_2,\ldots,k_r),h) \in K\wr H$ with the product $(k_1,k_2,\ldots,k_r)h$ of $(k_1,k_2,\ldots,k_r)\in K^r$ and $h\in H$ (both viewed as elements of the group $K\wr H$).

The wreath product $K\wr H$ can be endowed with an action on $\Delta^r$ by letting
\[(\delta_1, \delta_2, \ldots, \delta_r)^{(k_1, k_2, \ldots, k_r) h} = \left( \delta_1^{k_1}, \delta_2^{k_2}, \ldots, \delta_r^{k_r} \right)^{h} = \left( \delta_{1h^{-1}}^{k_{1h^{-1}}}, \delta_{2h^{-1}}^{k_{2h^{-1}}}, \ldots, \delta_{rh^{-1}}^{k_{rh^{-1}}} \right) \,,\]
for all $(\delta_1, \delta_2, \ldots, \delta_r) \in \Delta^r, (k_1,k_2,\ldots,k_r) \in K^r$, and $h\in H$.
We call this action the \emph{product action of the wreath product $K\wr H$ on $\Delta^r$}.

We recall the condition for a wreath product endowed with product action to be primitive.

\begin{lem}[{{\cite[Lemma~2.7A]{DM}}}]
\label{lem:primitiveWreath}
	Let $K$ be a permutation group on $\Delta$ and let $H$ be a permutation group on $\Omega$. The wreath product $K \wr H$ endowed with the product action on $\Delta^r$ is primitive if and only if $H$ is transitive and $K$ is primitive but not regular.
\end{lem}

We now introduce some notation to deal with any subgroup $G$ of $\mathrm{Sym}(\Delta) \wr \mathrm{Sym}(\Omega)$ endowed with product action on $\Delta^r$.

By abuse of notation, we identify the set $\Delta$ with
\[ \left\lbrace \{\delta\}\times \Delta^{r-1} \,\middle\vert\, \delta \in \Delta \right\rbrace \]
via the mapping $\delta \mapsto \{\delta\}\times \Delta^{r-1}$.
We denote by $G_\Delta^\Delta$ the permutation group that $G_\Delta$ induces on $\Delta$, that is,
\[G_\Delta^\Delta \cong G_\Delta / G_{(\Delta)} \,.\]
(Recall that $G_{(\Delta)}$ denotes the pointwise stabilizer of $\Delta$.)

Moreover, recalling that every element of $G$ can be written uniquely as $gh$, for some $g\in\Sym(\Delta)^r$ and some $h \in \Sym(\Omega)$, we can define the group homomorphism
\[\psi: G \to \Sym(\Omega), \quad gh \mapsto h \,.\]
This map defines a new permutational representation of $G$ acting on $\Omega$. We denote by $G^{\Omega}$ the permutation group corresponding to the faithful action that $G$ defines on $\Omega$, that is,
\[G^{\Omega} \cong G / \ker (\psi) \,.\]

Recall that a primitive group $G$, according to the O'Nan--Scott classification (see, for instance, \cite[III$(b)(i)$]{LiebeckPraegerSaxl1988}), is said to be of \emph{product action type} if there exists a transitive group $H\le \Sym(\Omega)$ and a primitive almost simple group $K\le \Sym(\Delta)$ with socle $T$ such that, for some integer $r\ge 2$, \[T^r\le G\le K \wr H \,,\]
where $T^r$ is the socle of $G$, thus contained in the base group $K^r$.
A detailed description of primitive groups of product action type was given by L.~G.~Kov\'acs in \cite{Kovacs}.

\begin{remark}
	By \cite[Theorem~1.1~$(b)$]{PraegerSchneider}, a group $G$ of product action type is permutationally isomorphic to a subgroup of $G^\Delta_\Delta \wr G^\Omega$. Therefore, up to a conjugation in $\Sym(\Delta^r)$, the group $K$ can always be chosen as $G^\Delta_\Delta$, and $H$ as $G^\Omega$.
\end{remark}
		
\subsection{Groups acting on digraphs}	
We give a short summary of standard notations for digraphs and graphs.

If a subgroup $G\le \mathrm{Aut}(\vG)$ is primitive on $V\vG$, we say that $\vG$ is \emph{$G$-vertex-primitive}. In a similar way, if $G$ is transitive on $A\vG$, we say that $\vG$ is \emph{$G$-arc-transitive}. The analogue notions can be defined for graphs, and when $G=\Aut(\vG)$ we drop the prefix $G$.

For any vertex $v\in \V\vG$, we denote by $\vG(v)$ its \emph{out-neighbourhood}, that is, the set of vertices $u\in \vG$ such that $(v,u)\in A\vG$. The size of the out-neighbourhood of a vertex $v$, $|\vG(v)|$, is called \emph{out-valency of $v$}. If $\vG$ is $G$-vertex-primitive, for some group $G$, then the out-valency in independent of the choice of the vertex $v$, thus we will refer to it as the \emph{out-valency of $\vG$}, in symbols $\val(\vG)$.
Whenever $\vG$ is a graph, \emph{neighbourhood} and \emph{valency} can be defined in the same way.

An \emph{orbital for $G$} is an orbit of $G$ in its induced action on $\Omega\times \Omega$. An \emph{orbital digraphs for $G$} is a digraph whose vertex-set is $\Omega$, and whose arc-set is an orbital for $G$. An example of orbital for $G$ is the {\em diagonal orbital} $(\omega,\omega)^G$, whose corresponding disconnected orbital graph is called {\em diagonal orbital graph}. We refer to \cite[Section~3.2]{DM} for further details.

Note that an orbital graph for $G$ is always $G$-arc-transitive, and, conversely, every $G$-arc-transitive digraph is an orbital graph for $G$. Furthermore, if $G\le \mathrm{Aut}(\vG)$ is a group of automorphism for a given digraph $\vG$, then $\vG$ is a union of orbitals for $G$ acting on $\V\vG$.

The number of distinct orbital digraphs for $G$ is called the \emph{permutational rank of $G$}. In particular, $2$-transitive permutation groups are precisely those of permutational rank $2$.

If $A\subseteq \Omega\times\Omega$ is an orbital for $G$, then so is the set $A^*=\{(\beta,\alpha) \mid (\alpha,\beta) \in A\}$. If $A=A^*$, then 
the orbital $A$ is called {\em self-paired}. Similarly, an orbital digraph is {\em self-paired} if its arc-set is a self-paired orbital. Note that any $G$-arc-transitive graph is obtained from a self-paired orbital digraph for $G$.

\section{Orbital digraphs for wreath products in product action}
\label{sec:preliminaryResults}
We are interested in reconstructing the orbital digraphs of a wreath product $K\wr H$ endowed with product action once the orbital digraphs of $K$ are known.

\begin{lem}
\label{lem:obritalPA}
	Let $K\wr H$ be a wreath product endowed with the product action on $\Delta^r$, and let
	\[\mG=\{ \vG_0, \vG_1, \ldots, \vG_k \}\]
	be the complete list of the orbital digraphs for $K$. Then any orbital digraph is a merged product action digraph of the form
	\[\MP\left( r,\mG,(j_1,j_2,\ldots,j_r)^H \right) \,,\]
	for a sequence of indices $(j_1,j_2,\ldots,j_r)\in X^r$, where $X=\{0,1,\ldots,k\}$.
\end{lem}
\begin{proof}
	Let $\vG$ be an orbital digraph for $K\wr H$. Suppose that $(u,v)\in A\vG$, where $u=(u_1,u_2,\ldots, u_r)$ and $v=(v_1,v_2,\ldots, v_r)$. We aim to compute the $K\wr H$-orbit of $(u,v)$, and, in doing so, proving that there is a sequence of indices $(j_1,j_2,\ldots,j_r)\in X^r$ such that
	\[A\vG = A\MP\left( r,\mG,(j_1,j_2,\ldots,j_r)^H \right) \,.\]
	
	We start by computing the $K^r$-orbit of $(u,v)$ (where by $K^r$ we refer to the base group of $K\wr H$). Since this action is componentwise, we obtain that
	\[(u,v)^{K^r}=A\left( \vG_{j_1} \times \vG_{j_2} \times \ldots \times \vG_{j_r}\right) \,,\]
	where $(u_i,v_i)$ is an arc of $\vG_{j_i}$ for all $i=1,2,\ldots,r$.
	
	The top group $H$ acts by permuting the components, so that
	\[(u,v)^{K\wr H}=\bigcup_{(j_1',j_2',\ldots,j_r') \in (j_1,j_2,\ldots,j_r)^H} A\left( \vG_{j_1'} \times \vG_{j_2'} \times \ldots \times \vG_{j_r'}\right)\]
	Therefore, the arc-sets of  $\vG$ and $\MP\left( r,\mG,(j_1,j_2,\ldots,j_r)^H \right)$ coincide.
	
	As their vertex-sets are both $\Delta^r$, the proof is complete.
\end{proof}

Now that we know how to build the orbital digraphs for a permutation group in product action, we ask ourselves what can we say about the orbital digraphs of its subgroups.

\begin{thm}
\label{thm:suborbits}
	Let $G\le \mathrm{Sym}(\Delta) \wr \mathrm{Sym}(\Omega)$ be a primitive group of product action type, and let $T$ be the socle of $G_\Delta^\Delta$. Suppose that $T$ and $G_\Delta^\Delta$ share the same orbital digraphs.
	Then the orbital digraphs for $G$ coincide with the orbital digraphs for $G_\Delta^\Delta \wr G^\Omega$, or, equivalently, for $T \wr G^\Omega$.
\end{thm}
\begin{proof}
	Since $G$ is a primitive group of product action type, we can suppose that $G$ is a subgroup of $G_\Delta^\Delta \mathop\mathrm{wr} G^\Omega$ with socle $T^r$, where $r=|\Omega|$. Further, we set $K = G_\Delta^\Delta$, $H=G^\Omega$. 
	
	As $G\le K\wr H$, the partition of $\Delta^r\times \Delta^r$ in arc-sets of orbital digraphs for $K\wr H$ is coarser than the one for $G$. Hence, our aim is to show that a generic orbital digraph for $K\wr H$ is also an orbital digraph for $G$.
	
	Let
	\[\mG=\{ \vG_0, \vG_1, \ldots, \vG_k \}\]
	be the complete list of orbital digraphs for $T$ acting on $\Delta$, and let $X=\{0,1,\ldots,k\}$. Observe that the set of orbital digraphs for $T^r$ can be identified with the Cartesian product of $r$ copies of $\mG$: indeed, we can bijectively map the generic orbital digraph $T^r$, say $\vG_{j_1}\times \vG_{j_2}\times \ldots \times \vG_{j_r}$, for some $(j_1,j_2,\ldots,j_r)\in X^r$, to the generic $r$-tuple of the Cartesian product $\mG^r$ of the form $( \vG_{j_1}, \vG_{j_2}, \ldots , \vG_{j_r} )$. This point of view explains why $H$ can act on the set of orbital digraphs for $T^r$ with its action of rank $r$.
	
	Observe that the set of orbital digraphs for $T^r$ is equal to the set of orbital digraphs for $K^r$. Moreover, $T^r$ is a subgroup of $G$, and $K^r$ is a subgroup of $K \wr H$. Thus the orbital digraphs for $G$ and for $K \wr H$ are obtained as a suitable unions of the elements of $\mG^r$.
	
	By Lemma~\ref{lem:obritalPA}, the orbital digraphs for $K \wr H$ are of the form
	\[\bigcup_{(j_1',j_2',\ldots,j_r') \in (j_1,j_2,\ldots,j_r)^H} \vG_{j_1'} \times \vG_{j_2'} \times \ldots \times \vG_{j_r'} \,,\]
	for some $(j_1,j_2,\ldots,j_r)\in X^r$. Aiming for a contradiction, suppose that
	\[\vG_{j_1} \times \vG_{j_2} \times \ldots \times \vG_{j_r} \quad \hbox{and} \quad \vG_{i_1} \times \vG_{i_2} \times \ldots \times \vG_{i_r}\]
	are two distinct orbital digraphs for $T^r$ that are merged under the action of top group $H$, but they are not under the action of $G$. The first portion of the assumption yields that there is an element $h\in H$ such that
	\[ \left(\vG_{j_1}\times \vG_{j_2}\times \ldots \times \vG_{j_r}\right) ^h =
	\vG_{i_1} \times \vG_{i_2} \times \ldots \times \vG_{i_r} \,.\]
	By definition of $H=G^\Omega$, there is an element in $G$ of the form
	\[(g_1,g_2,\ldots,g_r)h\in G.\]
	Recalling that, for any $i=1,2,\ldots,r$, $g_i\in K$, we get
	\[ \left(\vG_{j_1}\times \vG_{j_2}\times \ldots \times \vG_{j_r}\right) ^{(g_1,g_2,\ldots,g_r) h} =
	\vG_{i_1} \times \vG_{i_2} \times \ldots \times \vG_{i_r} \,.\]
	Therefore, the merging among these orbital graphs is also realised under the action of $G$, a contradiction.
	
	By the initial remark, the proof is complete.
\end{proof}

\section{Daily specials}
\label{sec:con}
The aim of this section is to give a descriptions of the digraphs appearing in Theorem~\ref{thm:main}.

\subsection{Generalised Hamming graphs}
\label{sec:genHam}

In this section, we clarify Remark~\ref{rem:orbHam} and we compute the relative fixity of the generalised Hamming graphs.

\begin{lem}
	\label{lem:Ham}
	Let $H\le \Sym(r)$ be a transitive permutation group, let $G=\Alt(\Delta)\wr H$ endowed with the product action on $\Delta^r$, and let $\vG$ be a digraph with vertex-set $V\vG = \Delta^r$. Then $G\le \Aut(\Gamma)$ if and only if $\vG$ is a generalised Hamming graph $\mathbf{H}(r,m,\mJ)$,
	where $|\Delta|=m$ and $\mJ\subseteq \{0,1\}^r$ is $H$-invariant.
\end{lem}
\begin{proof}
	By applying Lemma~\ref{lem:obritalPA} and taking the union of the resulting orbital digraphs, we obtain the left-to-right direction of the equivalence.
	Let us now deal with the converse implication. Let $\vG=\mathbf{H}(r,m,\mJ)$,
	where $|\Delta|=m$ and $\mJ\subseteq \{0,1\}^r$ is $H$-invariant.
	By Construction~\ref{con:general} and Definition~\ref{de:genHam},
	\[\mathbf{H}(r,m,\mJ) = \bigcup\limits_{h\in H} \left( \bigcup\limits_{i=0}^b  \bK_m^{a+i} \times \bL_m^{b+c-i} \right)^h \,,\]
	for some non negative integers $a,b$ such that $a+b\le r$.
	As each component of the graphs in parenthesis is either $\bK_m, \bL_m$ or $\bK_m\cup\bL_m$, we have that \[\Alt(m)^r\le \Aut\left( \bigcup\limits_{i=0}^b  \bK_m^{a+i} \times \bL_m^{b+c-i} \right) \,.\]
	Moreover, as $\mJ$ is $H$-invariant, the action of rank $r$ that $H$ induces on $\Delta^r$ preserves the arc-set of $\mathbf{H}(r,m,\mJ)$. As $G$ is generated by $\Alt(m)^r$ and this $H$ in their actions on $\Delta^r$,
	this implies that $G\le \Aut(\vG)$, as claimed.
\end{proof}

Instead of directly computing the relative fixity of $\mathbf{H}(r,m,\mJ)$, we prove the following stronger result.
\begin{lem}
\label{lem:relFixHam}
	Let $K\wr H$ be a wreath product endowed with the product action on $\Delta^r$, and let $\vG$ be a digraph with vertex set $\Delta^r$. Suppose that $K\wr H\le \Aut(\vG)$. Then
	\[\RelFix(\vG) = 1 - \frac{\mu\left( \Aut(\vG) \cap \Sym(\Delta)^r \right)}{|V\vG|} \,.\]
	In particular, the relative fixity of a generalised Hamming graph is
	\[\RelFix\left( \mathbf{H}(r,m,\mJ) \right) = 1 - \frac{2}{m} \,.\]
\end{lem}
\begin{proof}
	Suppose that $|\Delta|=m$, then, by hypothesis,
	\[K \wr H \le \Aut(\vG) \le \Sym(m) \wr \Sym(r) \,.\]
	We claim that the automorphism that realizes the minimal degree must be contained in $\Aut(\vG)\cap \Sym(m)^r$ (where $\Sym(m)^r$ is the base group of $\Sym(m)\wr \Sym(r)$). Indeed, upon choosing an element of minimal degree in $K\times \{\id\}\times \ldots \{\id\}$ and a transposition from the top group in $\Sym(m)\wr \Sym(r)$, we obtain the inequalities
	\[\begin{split}
		\mu\left( \Aut(\vG)\cap \Sym(m)^r \right) &\le \mu(K) m^{r-1} 
		\\ &\le (m-1)m^{r-1}
		\\ &\le \min \left\{|\supp(g)| \mid g \in \Aut(\vG)\setminus \Sym(m)^r \right\}
	\end{split}\]
	This is enough to prove the first portion of the statement.
	
	In particular, to compute the relative fixity of $\mathbf{H}(r,m,\mJ)$, it is enough to look at the action of $\Sym(m)$ on a single component. Thus, upon choosing a transposition in $\Sym(m)\times \{\id\}\times \ldots \{\id\}$, we obtain
	\[\RelFix\left( \mathbf{H}(r,m,\mJ) \right) =  1 - \frac{2m^{r-1}}{m^r} = 1 - \frac{2}{m}\,. \qedhere\]
\end{proof}

\color{black}

\subsection{Distance-$i$ Johnson graphs}
\label{sec:mergedJordan}

The nomenclature dealing with possible generalizations of the Johnson graph is as lush as confusing. In this paper, we are adopting the one from \cite{JonesJaycay}.

Let $m,k,i$ be integers such that $m\ge 1$, $1\le k\le m$ and $0\le i\le k$. A \emph{distance-$i$ Johnson graph}, denoted by $\mathbf{J}(m,k,i)$ is a graph whose vertex-set is the family of $k$-subsets of $\{1,2,\ldots, m\}$, and such that two $k$-subsets, say $X$ and $Y$, are adjacent whenever $|X\cap Y|=k-i$. The usual Johnson graph is then $\mathbf{J}(m,k,1)$, and two subsets $X$ and $Y$ are adjacent in $\mathbf{J}(m,k,i)$ if and only if they are at distance-$i$ in $\mathbf{J}(m,k,1)$.

\begin{lem}
\label{lem:orbJohn}
	Let $m,k$ be two positive integers such that $m \ge 2k+2$. The orbital digraphs of $\Alt(m)$ and of $\Sym(m)$ in their action on $k$-subsets are the distance-$i$ Johnson graphs $\mathbf{J}(m,k,i)$, one for each choice of $i\in \{0, 1,\ldots, k\}$.
\end{lem}
\begin{proof}
	Suppose that two $k$-subsets $X$ and $Y$ are such that $(X,Y)$ is an arc of the considered orbital digraph and $|X\cap Y|=k-i$, for a nonnegative integer $i\le k$. Since $\Alt(m)$ is $(m-2)$-transitive and $2k \le m-2$, the $\Alt(m)$-orbit of the arc $(X,Y)$ contains all the pairs $(Z,W)$, where $Z$ and $W$ are $k$-subsets with $|Z\cap W|=k-i$. Therefore, the statement is true for the alternating group. The same proof can be repeated \emph{verbatim} for $\Sym(m)$. 
\end{proof}

\begin{lem}
\label{lem:fixJohn}	
	Let $m,k,i$ be three positive integers such that $m \ge 2k+2$ and $i\ne k$.
	Then the relative fixity of the distance-$i$ Johnson graphs $\mathbf{J}(m,k,i)$ is
	\[\RelFix(\mathbf{J}(m,k,i)) = 1 - \frac{2k(m-k)}{m(m-1)}\,.\]
\end{lem}
\begin{proof}
	Under our assumption, by \cite[Theorem~2~$(a)$]{Jones2005}, the automorphism group of $\mathbf{J}(m,k,i)$ is $\Sym(m)$ in its action on $k$ subsets.
	Its minimal degree is achieved by any transposition (see \cite[Section~1]{GuralnickMagaard}), where
	\[\mu\left( \Sym(m) \right) = 2 \binom{m-2}{k-1}\,.\]
	Hence, we find that
	\[\RelFix(\mathbf{J}(m,k,i)) = 1 - \frac{2k(m-k)}{m(m-1)}\,. \qedhere\]
\end{proof}

\subsection{Squashed distance-$i$ Johnson graphs}
\label{sec:quotJordan}

A usual construction in the realm of distance transitive graphs consist in obtaining smaller example starting from a distance transitive graph and identifying vertices at maximal distance. We need to apply this idea to a family of generalised Johnson graphs.

Consider the distance-$i$ Johnson graph $\mathbf{ J}(2m,m,i)$, for some integers $m$ and $i$, with $m$ positive and $0\le i\le m$. We say that two vertices of ${\mathbf J}(2m,m,i)$ are \emph{disjoint} if they have empty intersection as $m$-subset. Observe that being disjoint is an equivalence relation, and our definition coincides with the usual notion of antipodal for ${\mathbf J}(2m,m,1)$ seen as a metric space. We can build a new graph $\mathbf{QJ}(2m,m,i)$ whose vertex-set is the set of equivalence classes of the disjoint relation, and such that, if $[X]$ and $[Y]$ are two generic vertices, then $([X],[Y])$ is an arc in $\mathbf{QJ}(2m,m,i)$ whenever there is a pair of representatives, say $X'\in[X]$ and $Y'\in[Y]$, such that $(X',Y')$ is an arc in ${\mathbf J}(2m,m,i)$. We call $\mathbf{QJ}(2m,m,i)$ an \emph{squashed distance-$i$ Johnson graph}.

Observe that $\mathbf{J}(2m,m,i)$ is a regular double cover of $\mathbf{QJ}(2m,m,i)$. Furthermore, $\mathbf{QJ}(2m,m,i)$ and $\mathbf{QJ}(2m,m,m-i)$ are isomorphic graphs, thus we can restrict the range of $i$ to $\{0,1,\ldots, \lfloor m/2 \rfloor\}$.

\begin{lem}
\label{lem:orbSq}
	Let $m\ge 3$ be an integer. The orbital digraphs of $\Alt(2m)$ and of $\Sym(2m)$ in their primitive actions with stabilizer $(\Sym(m) \wr C_2)\cap \Alt(2m)$ and $\Sym(m) \wr C_2$ respectively are the squashed distance-$i$ Johnson graphs $\mathbf{J}(m,k,i)$, one for each choice of $i\in \{0, 1,\ldots, \lfloor m/2\rfloor\}$.
\end{lem}
\begin{proof}
	To start, we note that the set $\Omega$ on which the groups are acting can be identified with the set of partitions of the set $\{1,2,\ldots,2m\}$ with two parts of equal size $m$. Suppose that $\{X_1,X_2\}$ and $\{Y_1,Y_2\}$ are two such partitions and that $\left(\{X_1,X_2\},\{Y_1,Y_2\}\right)$ is an arc of the orbital digraph we are building, with
	\[\min\{|X_1\cap Y_1|,\,|X_1\cap Y_2|\} = m-i \,,\]
	for a nonnegative integer $i\le \lfloor m/2\rfloor$. To determine the image of $\left(\{X_1,X_2\},\{Y_1,Y_2\}\right)$ under the group action, it is enough to know the images of $X_1$ and $Y_2$, that is, of at most $2m-\lceil m/2\rceil \le 2m - 2$ distinct points. By the $(2m-2)$-transitivity of $\Alt(2m)$, the $\Alt(2m)$-orbit of $\left(\{X_1,X_2\},\{Y_1,Y_2\}\right)$ contains all the arc of the form $\left(\{Z_1,Z_2\},\{W_1,W_2\}\right)$, where $\{Z_1,Z_2\},\{W_1,W_2\}\in \Omega$ and
	\[\min\{|Z_1\cap W_1|,\,|Z_1\cap W_2|\} = m-i \,.\]
	To conclude, observe that $\Omega$ is the set of $m$-subsets of $\{1,2,\ldots,2m\}$ in which two elements are identified if they are disjoint, and that
	\[\min\{|X_1\cap Y_1|,\,|X_1\cap Y_2|\} = m-i \,,\]
	is the adjacency condition in an squashed distance-$i$ Johnson graph. As in Lemma~\ref{lem:orbJohn}, the same reasoning can be exteneded to $\Sym(2m)$. Therefore, the orbital digraphs of $\Alt(2m)$ and of $\Sym(2m)$ in these primitive actions are the squashed distance-$i$ Johnson graphs $\mathbf{QJ}(2m,m,i)$, for some $i \in \{0,1,\ldots,\lfloor m/2 \rfloor\}$.
\end{proof}

\begin{lem}
\label{lem:fixSq}	
	Let $m,i$ be two positive integers such that $m \ge 3$ and $i\ne \lfloor m/2 \rfloor$.
	Then the relative fixity of the distance-$i$ Johnson graphs $\mathbf{QJ}(2m,m,i)$ is
	\[\RelFix(\mathbf{QJ}(2m,m,i)) = 1 - \frac{2k(m-k)}{m(m-1)} \,.\]
\end{lem}
\begin{proof}
	Consider $\mathbf{J}(2m,m,i)$, the regular double covering of $\mathbf{QJ}(2m,m,i)$.
	In view of \cite[Theorem~2~$(e)$]{Jones2005}, the automorphism group of $\mathbf{J}(2m,m,i)$ is $\Sym(2m)\times \Sym(2)$, where the central involution swaps pairs
	disjoint vertices. It follows that the automorphism group of $\mathbf{QJ}(2m,m,i)$ is $\Sym(2m)$. Now, we can immediately verify that the stabilizer of the vertex $\{\{1,2,\ldots,m\},\{m+1,m+2,\ldots,2m\}\}$ is $\Sym(m)\wr C_2$.
	The minimal degree of the primitive action of $\Sym(2m)$ with stabilizer $\Sym(m)\wr C_2$ is
	\[\mu\left( \Sym(2m) \right) = \frac{1}{4}\left( 1 + \frac{1}{2m-1} \right) \frac{(2m)!}{m!^2}\]
	(see \cite[Theorem~4]{BurnessGuralnick}). Thus, we find that
	\[\RelFix(\mathbf{QJ}(2m,m,i)) = \frac{1}{2}\left( 1 - \frac{1}{2m-1} \right) \,. \qedhere\]
\end{proof}
\color{black}

\subsection{Strongly regular graphs} \label{sec:SRG} We list all the strongly regular graphs appearing as $\vG_1$ in Theorem~\ref{thm:main}~$(c)$. We divide them according to the socle $L$ of the almost simple group that acts on them. Further, the present enumeration corresponds to the one of the groups that act on these graphs as listed in (the soon to be enunciated) Theorem~\ref{thm:GurMag}~$(e)$.
\begin{enumerate}[$(i)$]
	\item $L=U_4(q)$, $q \in \{ 2, 3 \}$, acting on totally singular $2$-dimensional subspaces of the natural module, two vertices of $\Gamma$ are adjacent if there is a third $2$-dimensional subspace that intersect both vertices in a $1$-dimensional subspace (see \cite[Section~2.2.12]{SRG});
%	if $q=2$, then
%	\begin{alignat*}{2}
%		& v = 27 \qquad \qquad \qquad
%		&& \theta_1 = -5
%		\\
%		& k = 10
%		&& M_1 = 6  
%		\\
%		&\lambda = 1
%		&&\theta_2 = 1
%		\\
%		&\mu = 5
%		&&M_2 = 20
%		\\
%		& \RelFix(\vG)= \frac{7}{27} &&
%	\end{alignat*}
%	
%	and, if $q=3$, then\begin{alignat*}{2}
%		& v = 112
%		&& \theta_1 = -10
%		\\
%		& k = 30
%		&& M_1 = 21 
%		\\
%		&\lambda = 2
%		&&\theta_2 = 2
%		\\
%		&\mu = 10
%		&&M_2 = 90
%		\\
%		& \RelFix(\vG)= \frac{11}{56} &&
%	\end{alignat*}
	\item $L=\Omega_{2m+1}(3), m\ge 2$, acting on the singular points of the natural module, two vertices of $\Gamma$ are adjacent if they are orthogonal (see \cite[Theorem~2.2.12]{SRG});
%	and 
%	\begin{alignat*}{2}
%		& v= \frac{1}{2}(3^{2m}-1) \qquad \qquad \qquad
%		&& \theta_1 = 3^{m-1} -1 \\
%		& k= \frac{3}{2}(3^{2m-2}-1)
%		&& M_1=\frac{3}{4}(3^m-1)(3^{m-1}+1) \\
%		&\lambda = \frac{9}{2}(3^{2m-4}-1) + 2
%		&&\theta_2= -3^{m-1} -1 \\
%		&\mu = \frac{1}{2}(3^{2m-2}-1)
%		&&M_2=\frac{3}{4}(3^{m-1}-1)(3^m+1)\\
%		& \RelFix(\vG) = 1 - \frac{2 \,\cdot\,3^{m-1}}{3^m+1} &&
%	\end{alignat*}

	\item $L=\Omega_{2m+1}(3), m\ge 2$, acting on the nonsingular points of the natural module, two vertices of $\Gamma$ are adjacent if the line that connects them is tangent to the quadric where the quadratic form vanishes (see \cite[Section~3.1.4]{SRG});
%	, and 
%	\begin{alignat*}{2}
%		& v= \frac{3^m}{2}(3^m-1) \qquad \qquad \qquad
%		&& \theta_1 = 3^{m-1} -1 \\
%		& k= (3^{m-1}-1)(3^m+1)
%		&& M_1= \frac{1}{2}3^{2m} -\frac{3}{4}(3^{2m-1}-1) -1 \\
%		&\lambda = 2(3^{2m-2}-3^{m-1}-1)
%		&&\theta_2= -3^{m-1} -1 \\
%		&\mu = 2\,\cdot\,3^{m-1}(3^{m-1}-1)
%		&&M_2= \frac{3}{4}(3^{2m-1}-1) - \frac{1}{2}3^m \\
%		& \RelFix(\vG) = 1 - \frac{2 (3^{2m-1} - 2\,\cdot\,3^{m-1}-1)}{3^m(3^m-1)} &&
%	\end{alignat*}

	\item $L=\mathrm{P}\Omega_{2m}^\varepsilon(2), \varepsilon\in \{+,-\}, m\ge 3$, acting on the singular points of the natural module, two vertices of $\Gamma$ are adjacent if they are orthogonal (see \cite[Theorem~2.2.12]{SRG});
%	, if $\varepsilon=+$, then
%	\begin{alignat*}{2}
%		& v= (2^m-1)(2^{m-1}+1) \qquad \qquad \qquad
%		&& \theta_1 = 2^{m-1} -1 \\
%		& k= 2(2^{m-1}-1)(2^{m-2}+1)
%		&& M_1=\frac{2}{3}(2^m-1)(2^{m-2}+1) \\
%		&\lambda = 4(2^{m-2}-1)(2^{m-3}+1) + 1
%		&&\theta_2= -2^{m-2} -1 \\
%		&\mu = (2^{m-1}-1)(2^{m-2}+1)
%		&&M_2=\frac{4}{3}(2^{2m-2}-1)\\
%		&\RelFix(\Gamma) = 1 - \frac{2^{m-1}}{2^m-1} && 
%	\end{alignat*}
%	meanwhile, if $\varepsilon=-$, then
%	\begin{alignat*}{2}
%		& v= 2^{2m-2}-1 \qquad \qquad \qquad
%		&& \theta_1 = 2^{m-2} -1 \\
%		& k= 2(2^{2m-4}-1)
%		&&M_1= 4(2^{m-1}-1)(2^{m-2}+1) \\
%		&\lambda = 4(2^{2m-6}-1) + 1
%		&&\theta_2= -2^{m-2} -1 \\
%		&\mu = 2^{2m-4}-1
%		&&M_2= 4(2^{m-2}-1)(2^{m-1}+1)\\
%		&\RelFix(\Gamma) = 1 - \frac{2^{m-1}}{2^m+1} &&
%	\end{alignat*}
	
	\item $L=\mathrm{P}\Omega_{2m}^\varepsilon(2), \varepsilon\in \{+,-\}, m\ge 2$, acting on the nonsingular points of the natural module, two vertices of $\Gamma$ are adjacent if they are orthogonal (see \cite[Section~3.1.2]{SRG});
%	, and
%	\begin{alignat*}{2}
%		& v= 2^{m-1}(2^m-\varepsilon) \qquad \qquad \qquad
%		&& \theta_1 = \varepsilon2^{m-2} -1 \\
%		& k= 2^{2m-2}-1
%		&&M_1= \frac{4}{3}(2^{m-2}-1) \\
%		&\lambda = 2(2^{2m-4}-1)
%		&&\theta_2= -\varepsilon2^{m-1} -1 \\
%		&\mu = 2^{m-2}(2^{m-1}+\varepsilon)
%		&&M_2= \frac{1}{3}(2^{m-2}-\varepsilon)(2^{m-1}-\varepsilon)\\
%		&\RelFix(\Gamma) = 1 - \frac{2^{m-1}-\varepsilon}{2^m-\varepsilon} &&
%	\end{alignat*}
	
	\item $L=\mathrm{P}\Omega_{2m}^+(3), m\ge 2$ acting on the nonsingular points of the natural module, two vertices of $\Gamma$ are adjacent if they are orthogonal (see \cite[Section~3.1.3]{SRG});
%	, and
%	\begin{alignat*}{2}
%		& v= \frac{3^{m-1}}{2}(3^m-1) \qquad \qquad \qquad
%		&& \theta_1 = 3^{m-1} \\
%		& k= \frac{3^{m-1}}{2}(3^{m-1}-1)
%		&&M_1= \frac{1}{8}(3^m-1)(3^{m-1}-1) \\
%		&\lambda = \frac{3^{m-2}}{2}(3^{m-1}+1)
%		&&\theta_2= -3^{m-2} \\
%		&\mu = \frac{3^{m-1}}{2}(3^{m-2}-1)
%		&&M_2= \frac{9}{8}(3^{2m-2}-1)\\
%		&\RelFix(\Gamma) = 1 - \frac{2 (3^{m-1}-1)}{3^m-1} &&
%	\end{alignat*}
	
	\item $L=\mathrm{P}\Omega_{2m}^-(3), m\ge 3$ acting on the singular points of the natural module, two vertices of $\Gamma$ are adjacent if they are orthogonal (see \cite[Theorem~2.2.12]{SRG});
%	, and
%	\begin{alignat*}{2}
%		& v= \frac{1}{2}(3^{2m-2}-1) \qquad \qquad \qquad
%		&& \theta_1 = 3^{m-2} -1 \\
%		& k= \frac{3}{2}(3^{2m-4}-1)
%		&&M_1= \frac{3}{4}(3^{m-1}-1)(3^{m-2}+1) \\
%		&\lambda = \frac{9}{2}(2^{2m-6}-1) + 2
%		&&\theta_2= -3^{m-2} -1 \\
%		&\mu = \frac{1}{2}(3^{2m-4}-1)
%		&&M_2= \frac{3}{4}(3^{m-2}-1)(3^{m-1}+1)\\
%		&\RelFix(\Gamma) = 1 - \frac{2 \,\cdot\, 3^{m-1}}{3^m+1} &&
%	\end{alignat*}
	
	\item $L=\mathrm{P}\Omega_{2m}^-(3), m\ge 2$  acting on the nonsingular points of the natural module, two vertices of $\Gamma$ are adjacent if they are orthogonal (see \cite[Section~3.1.3]{SRG}).
%	, and
%	\begin{alignat*}{2}
%		& v= \frac{3^{m-1}}{2}(3^m+1) \qquad \qquad \qquad
%		&& \theta_1 = -3^{m-1} \\
%		& k= \frac{3^{m-1}}{2}(3^{m-1}+1)
%		&&M_1= \frac{1}{8}(3^m+1)(3^{m-1}+1) \\
%		&\lambda = \frac{3^{m-2}}{2}(3^{m-1}-1)
%		&&\theta_2= 3^{m-2} \\
%		&\mu = \frac{3^{m-1}}{2}(3^{m-2}+1)
%		&&M_2= \frac{9}{8}(3^{2m-2}-1)\\
%		&\RelFix(\Gamma) = 1 - \frac{2 (3^{2m-2}-1)}{3^{m-1}(3^m+1)} &&
%	\end{alignat*}
\end{enumerate}

Table~\ref{table} collects the usual parameters of a strongly regular graph, $(v,d,\lambda,\mu)$, and their relative fixity. Recall that $v$ is the number of vertices, $d$ is the valency of the graph, $\lambda$ is the number of common neighbours between two adjacent vertices, and $\mu$ is the number of common neighbours between two nonadjacent vertices.
As $\mu(G)$ can be found in \cite[Theorem~4]{BurnessGuralnick}, the relative fixity is computed as
\[\RelFix(\vG) = 1 - \frac{\mu(G)}{v} \,,\]

\begin{scriptsize}
	\begin{table}[t]
	\begin{center}
	\rowcolors{2}{white}{gray!25}
		\begin{tabular}{| c c | c c c c c | c |}
			\hline
			\phantom{$\dfrac{1^1}{1_{1_1}}$}\hspace{-6mm}
			& Socle & $v$ & $d$ & $\lambda$ & $\mu$ & $\RelFix$ & Comments\\
			\hline
			
			%\hline
			\phantom{$\dfrac{1^1}{1_{1_1}}$}\hspace{-6mm}
			$(i)$ & $U_4(2)$ & $27$ & $10$ & $1$ & $5$ & $\dfrac{7}{27}$ & \\
			
			%\hline
			\phantom{$\dfrac{1^1}{1_{1_1}}$}\hspace{-6mm}
			& $U_4(3)$ & $112$ & $30$ & $2$ & $10$ & $\dfrac{11}{56}$ & \\
			
			%\hline
			\phantom{$\dfrac{1^1}{1_{1_1}}$}\hspace{-6mm}
			$(ii)$ & $\Omega_{2m+1}(3)$ & $\dfrac{1}{2}(9a - 1)$ & $\dfrac{3}{2}(a^2-1)$ &   $\dfrac{1}{2}(a^2-9) + 2$ &  $ \dfrac{1}{2}(a^2-1)$ & $\dfrac{a+1}{3a+1}$ & $a=3^{m-1}$ \\
			
			%\hline
			\phantom{$\dfrac{1^1}{1_{1_1}}$}\hspace{-6mm}
			$(iii)$ & $\Omega_{2m+1}(3)$ & $\dfrac{3a}{2}(3a-1)$ & $(a-1)(3a+1)$ & $2(a^2-a-1)$ & $ 2a(a-1) $ & $\dfrac{3a^2+a+1}{3a(3a-1)}$ &  %$a=3^{m-1}$ 
			\\
			
			%\hline
			\phantom{$\dfrac{1^1}{1_{1_1}}$}\hspace{-6mm}
			$(iv)$ & $\mathrm{P}\Omega_{2m}^+(2)$ & $(4b-1)(2b-1)$ & $2(2b-1)(b+1)$ & $(2b-2)(b-2)+1$ & $(2b-1)(b+1)$ & $\dfrac{b-1}{2b-1}$ & $b=2^{m-2}$\\
			
			%\hline
			\phantom{$\dfrac{1^1}{1_{1_1}}$}\hspace{-6mm}
			& $\mathrm{P}\Omega_{2m}^-(2)$ & $4b^2-1$ & $2(b^2-1)$ & $b^2-3$ & $ b^2-1 $ & $\dfrac{2b+1}{4b+1}$ & %$b=2^{m-2}$
			\\
			
			%\hline
			\phantom{$\dfrac{1^1}{1_{1_1}}$}\hspace{-6mm}
			$(v)$ & $\mathrm{P}\Omega_{2m}^\varepsilon(2)$ & $2b(4b-\varepsilon)$ & $4b^2-1$ & $2(b^2-1)$ & $ b(2b+\varepsilon) $ & $\dfrac{2b}{4b - \varepsilon}$ & $\varepsilon = \pm 1$ %$b=2^{m-2}$
			\\
			
			%\hline
			\phantom{$\dfrac{1^1}{1_{1_1}}$}\hspace{-6mm}
			$(vi)$ & $\mathrm{P}\Omega_{2m}^+(3)$ & $
			\dfrac{3c}{2}(9c-1)$ & $\dfrac{3c}{2}(3c-1)$ & $\dfrac{c}{2}(3c-1)$ & $\dfrac{3c}{2}(c-1)$ & $\dfrac{3(c+1)}{9c - 1}$ & $c=3^{m-2}$\\
			
			%\hline
			\phantom{$\dfrac{1^1}{1_{1_1}}$}\hspace{-6mm}
			$(vii)$ & $\mathrm{P}\Omega_{2m}^-(3)$ & $\dfrac{1}{2}(9c^2-1)$ & $\dfrac{3}{2}(c^2-1)$ & $\dfrac{1}{2}(c^2-9)+2$ & $\dfrac{1}{2}(c^2-1)$ & $\dfrac{3c+1}{9c+1}$ & %$c=3^{m-2}$
			\\
			
			%\hline
			\phantom{$\dfrac{1^1}{1_{1_1}}$}\hspace{-6mm}
			$(viii)$ & $\mathrm{P}\Omega_{2m}^-(3)$ & $\dfrac{3c}{2}(9c+1)$ & $\dfrac{3c}{2}(3c+1)$ & $\dfrac{c}{2}(3c-1)$ & $\dfrac{3c}{2}(c+1)$ & $\dfrac{9c^2+3c-2}{3c(9c+1)}$ & %$c=3^{m-2}$
			\\
			
			\hline
			%\hline
		\end{tabular}
	\end{center}
	{\caption{Parameters of strongly regular graphs with large fixity.}
	\label{table}}
	\end{table}
\end{scriptsize}

\section{Proof of Theorem~\ref{thm:main}}
\label{sec:proof}
The primitive permutation groups we are concerned with were classified by T.~Burness and R.~Guralnick in \cite{BurnessGuralnick}. We report their result here. For the sake of our proof, we explicitly write the permutational rank of the almost simple groups of Lie type. This information can be easily obtained combining the complete list of $2$-transitive finite permutation groups, first described by P.~J.~Cameron in \cite[Section~5]{Cameron2tr}, and the complete list of classical finite permutation groups of permutational rank $3$, compiled by W.~M.~Kantor and R.~A.~Liebler in \cite[Theorem~1.1]{KantorLieber}.

\begin{thm}[\cite{BurnessGuralnick}, Theorem~4]
	\label{thm:GurMag}
	Let $G$ be a permutation group with
	\[\mu(G) < \frac{2n}{3} \,.\]
	Then one of the following holds:
	\begin{enumerate}[$(a)$]
		\item $\Alt(m) \le G \le \Sym(m)$, for some $m\ge 3$, in its action on $k$-subsets, for some $k< m/2$;
		\item $G=\Sym(2m)$, for some $m\ge 2$, in its primitive action with stabilizer $G_\alpha = \Sym(m)\wr C_2$;
		\item $G = M_{22}:2$ in its primitive action of degree $22$ with stabilizer $G_\alpha = \mathrm{L}_3(4).2_2$;
		\item $G$ is an almost simple group of socle $L$ and permutational rank $2$, and one of the following occurs:
		\begin{enumerate}[$(i)$]
			\item $L=\mathrm{L}_m(2)$, $m\ge 3$, in its natural action;
			\item $L=\mathrm{L}_m(3)$, $m\ge 3$, in its natural action, and $G$ contains an element of the form $(-I_{n-1},I_1)$;
			\item $L=\mathrm{Sp}_{2m}(2)$, $m\ge 3$, in its action on the singular points of the natural module;
			\item $L=\mathrm{Sp}_{2m}(2)$, $m\ge 3$, in its action on the right cosets of $\SOminus(2)$;
			\item $L=\mathrm{Sp}_{2m}(2)$, $m\ge 3$, in its action on the right cosets of $\SOplus(2)$;
		\end{enumerate}
		\item $G$ is an almost simple group of socle $L$ and permutational rank $3$, and one of the following occurs:
		\begin{enumerate}[$(i)$]
			\item $L=U_4(q)$, $q\in\{2,3\}$, in its primitive action on totally singular $2$-dimensional subspaces, and $G$ contains the graph automorphism $\tau$;
			\item $L=\Omega_{2m+1}(3)$ in its action on the singular points of the natural module, and $G$ contains an element of the form $(-I_{2m},I_1)$ with a $+$-type $(-1)$-eigenspace;
			\item $L=\Omega_{2m+1}(3)$ in its action on the nonsingular points of the natural module whose orthogonal complement is an orthogonal space of $-$-type, and $G$ contains an element of the form $(-I_{2m},I_1)$ with a $-$-type $(-1)$-eigenspace;
			\item $L=\mathrm{P}\Omega_{2m}^\varepsilon(2)$, $\varepsilon\in \{+,-\}$, in its action on the singular points on the natural module, and $G=\mathrm{SO}_{2m}^\varepsilon(2)$;
			\item $L=\mathrm{P}\Omega_{2m}^\varepsilon(2)$, $\varepsilon\in \{+,-\}$, in its action on the nonsingular points on the natural module, and $G=\mathrm{SO}_{2m}^\varepsilon(2)$;
			\item $L=\mathrm{P}\Omega_{2m}^+(3)$ in its action on the nonsingular points on the natural module, and $G$ contains an element of the form $(-I_{2m-1},I_1)$ such that the discriminant of the $1$-dimensional $1$-eigenspace is a nonsquare;
			\item $L=\mathrm{P}\Omega_{2m}^-(3)$ in its action on the singular points on the natural module, and $G$ contains an element of the form $(-I_{2m-1},I_1)$;
			\item $L=\mathrm{P}\Omega_{2m}^-(3)$ in its action on the nonsingular points on the natural module, and $G$ contains an element of the form $(-I_{2m-1},I_1)$ such that the discriminant of the $1$-dimensional $1$-eigenspace is a square;
		\end{enumerate}
		\item $G\le K \wr \Sym(r)$ is a primitive group of product action type, where $K$ is a permutation group appearing in points $(a)-(e)$, the wreath product is endowed with the product action, and $r\ge 2$;
		\item $G$ is an affine group with a  regular normal socle $N$, which is an elementary abelian $2$-subgroup.
	\end{enumerate}
\end{thm}

\begin{proof}[Proof of Theorem~\ref{thm:main}]
The proof is split in two independent chunks. 
First, we prove that every vertex-primitive digraph of relative fixity exceeding $\frac{1}{3}$ belongs to one of the families appearing in Theorem~\ref{thm:main}. 
Then, we tackle the problem of computing the relative fixities of the graphs appearing in Theorem~\ref{thm:main}, thus showing that they indeed all have relative fixity larger than $\frac{1}{3}$.

Assume that $\vG$ is a digraph on $n$ vertices with at least one arc and with $\RelFix(\vG) > \frac{1}{3}$ such that $G=\Aut(\vG)$ is primitive.
If $\vG$ is disconnected, then the primitivity of $G$ implies that $\vG \cong \bL_n$. Hence we may assume that $\vG$ is connected.
Moreover, $\RelFix(\vG) > \frac{1}{3}$ implies that $\mu(G) < \frac{2n}{3}$. Hence $G$ is one of the groups 
determined in \cite{BurnessGuralnick} and described in Theorem~\ref{thm:GurMag}.

Suppose that $G$ is an almost simple group. Then $G$ is one of the groups appearing in parts $(a)-(e)$ of Theorem~\ref{thm:GurMag}. Since any $G$-vertex-primitive group is a union of orbital digraphs for $G$, the digraphs arising from these cases will be merged product action digraphs $\MP(1,\mG,\mJ)$ (see Remark~\ref{rem:ex}).
Hence, our goal is to consider these almost simple groups in turn and compile their list of orbitals digraphs $\mG$.

Let $G$ be a group as described in Theorem~\ref{thm:GurMag}~$(a)$. Lemma~\ref{lem:orbJohn} states the orbital digraphs for $G$ are the distance-$i$ Johnson graph $\mathbf{J}(m,k,i)$.

Assume that $k=1$, that is, consider the natural action of either $\Alt(m)$ or $\Sym(m)$ of degree $m$. Since this action is $2$-transitive, their set of orbital digraphs is $\mG=\{\bL_m,\bK_m\}$. In particular, $\MP(1,\mG,\mJ)=\mathbf{H}(1,m,\mJ)$. This case exhausts the generalized Hamming graphs with $r=1$, which appear in Theorem~\ref{thm:main}~$(i)$. Therefore, in view of Remark~\ref{rem:isomHam}, for as long as we suppose $r=1$, we can also assume that $\mJ$ is a non-Hamming homogeneous set. Observe $m\ge 4$, otherwise, we go against our assumption on the relative fixity.

Going back to distance-$i$ Johnson graphs, to guarantee that $\mJ$ is non-Hamming, we have to take $k\ge 2$. Thus, 
\[\mG=\left\{\mathbf{J}(m,k,i) \mid i\in \{0,1,\ldots,k\}\right\} \,,\]
which corresponds to Theorem~\ref{thm:main}~$(ii)(a)$.

Let $G=\Sym(2m)$ be a permutation group from Theorem~\ref{thm:GurMag}~$(b)$. If $m=2$, the degree of $G$ is $3$, and the relative fixity of any action of degree $3$ can either be $0$ or $\frac{1}{3}$. Hence, we must suppose that $m\ge 3$: by Lemma~\ref{lem:orbSq}, the orbital digraphs for $G$ are the squashed distance-$i$ Johnson graph $\mathbf{QJ}(2m,m,i)$. We obtain that
\[\mG=\left\{\mathbf{QJ}(2m,m,i) \mid i\in \{0,1,\ldots,\lfloor m/2\rfloor\}\right\}\,,\]
as described in Theorem~\ref{thm:main}~$(ii)(b)$.

Let $G=M_{22}:2$ in the action described in Theorem~\ref{thm:GurMag}~$(c)$.
Consulting the list of all the primitive groups of degree $22$ in \textsc{Magma} \cite{MR1484478} (which is based on the list compiled in \cite{CouttsQuickRoney-Dougal}), we realize that they are all $2$-transitive. Hence, the set of orbital digraphs is $\mG = \{\bK_{22},\bL_{22}\}$. In particular, all the graphs are generalised Hamming graphs.

Let $G$ be an almost simple of Lie type appearing in Theorem~\ref{thm:GurMag}~$(d)$. Since all these groups are $2$-transitive with a $2$-transitive socle $L$, their orbital digraphs are either $\bK_m$ or $\bL_m$, where $m\ge 7$ is the degree of $G$. Once again, we obtain only generalise Hamming graphs.

Let $G$ be an almost simple of Lie type described in Theorem~\ref{thm:GurMag}~$(e)$. Any group of permutational rank $3$ defines two nondiagonal orbital digraphs, and, as such digraphs are arc-transitive and one the complement of the other, they are strongly regular digraphs (see, for instance, \cite[Section~1.1.5]{SRG}). The set of orbital digraphs is of the form $\mG=\{\bL_m,\vG_1,\vG_2\}$, where we listed the possible $\vG_1$ in Section~\ref{sec:SRG}, and where $m=|V\vG_1|$. The graphs described in this paragraph appear in Theorem~\ref{thm:main}~$(ii)(c)$.

We have exhausted the almost simple groups from Theorem~\ref{thm:GurMag}. Hence, we pass to Theorem~\ref{thm:GurMag}~$(f)$. Suppose that $G\le K \wr \Sym(r)$ is a primitive group of product action type. We want to apply Theorem~\ref{thm:suborbits} to $G$. The only hypothesis we miss is that $T$ and $G_\Delta^\Delta$ share the same set of orbital digraphs.

We claim that $T$ and $K$ induces the same set of orbital digraphs. If $K$ is either alternating or symmetric, the claim follows from Lemmas~\ref{lem:orbJohn} and~\ref{lem:orbSq}. If $K$ is $2$-transitive, then we can observe that its socle $L$ is also $2$-transitive: the socle of $M_{22}:2$ is $T=M_{22}$ in its natural $3$-transitive action, while the socle $T$ of the almost simple groups of Lie type of rank $2$ is $2$-transitive by \cite[Section~5]{Cameron2tr}. In particular, $K$ and $T$ both have $\mG=\{\bL_m,\bK_m\}$ as their set of orbital graphs. Finally, suppose that $K$ is an almost simple group of permutational rank $3$. We have that its socle $T$ is also of permutational rank $3$ by \cite[Theorem~1.1]{KantorLieber}. Observe that, since any orbital digraph for $T$ is a subgraph of an orbital digraph for $G$, the fact that $G$ and $L$ both have permutational rank $3$ implies that they share the same set of orbital digraphs. Therefore, the claim is true.

By our claim together with the double inclusion
\[T \le G_\Delta^\Delta \le K \,,\]
we obtain that $T, G_\Delta^\Delta$ and $K$ all induce the same set of orbital digraphs. Therefore, we can apply Theorem~\ref{thm:suborbits} to $G$: we obtain that $G$ shares its orbital graphs with $T \wr G^{\Omega}$.

Therefore, all the $G$-vertex-primitive digraphs are union of orbital digraphs for $T \wr H$, with $T$ socle type of $G$ and $H$ transitive permutation group isomorphic to $G^\Omega$. In other words, we found all the graphs $\MP(r,\mG,\mJ)$ with $r\ge 2$ described in Theorem~\ref{thm:main}. (Recall that, by Definition~\ref{de:genHam}, among the graphs $\MP(r,\mG,\mJ)$, we find all the generalised Hamming graphs.)

Suppose that $G$ is an affine group with a  regular normal socle $N$, which is an elementary abelian $2$-subgroup. We have that $G$ can be written as the split extension $N : H$, where $H$ is a group of matrices that acts irreducibly on $N$. It follows that $G$ is $2$-transitive on $N$, hence, as $|N|\ge 4$, the graphs arising in this scenario are $\bL_{|N|}, \bK_{|N|}$ and $\bL_{|N|}\cup\bK_{|N|}$, which are generalised Hamming graphs.

We have completed the first part of the proof, showing that the list of vertex-primitive digraphs appearing in Theorem~\ref{thm:main} is exhaustive. As all the orbital digraphs in $\mG$ are actually graphs, the same property is true for the graphs in the list, as we have underlined in Remark~\ref{rem:graph}.

We can now pass to the second part of the proof, that is, we can now tackle the computation of relative fixities. We already took care of the generalised Hamming graphs in Lemma~\ref{lem:Ham}. Thus, we can suppose that $\vG$ is a merged product action graph $\MP(r,\mG,\mJ)$ appearing in Theorem~\ref{thm:main}~$(ii)$.

Suppose that $r=1$, that is, $\vG$ is a union of graphs for some primitive almost simple group $K$. (We are tacitely assuming that $K$ is maximal among the groups appearing in a given part of Theorem~\ref{thm:GurMag}.) In view of \cite[Theorem]{LIEBECK1987365}, we have that $K$ is a maximal subgroup of either $\Alt(|V\vG|)$ or $\Sym(|V\vG|)$. Therefore, there are just two options for $\Aut(\vG)$: either it is isomorphic to $K$ or it contains $\Alt(|V\vG|)$. In the latter scenario, as $\Alt(|V\vG|)$ is $2$-transitive on the vertices, we obtain that $\vG \in \{\bL_m,\bK_m,\bL_m\cup\bK_m\}$.
All those graphs are generalised Hamming graphs, against our assumption on $\vG$. Therefore, we have $K=\Aut(\vG)$. In particular, the relative fixity for $\vG$ are computed in Lemma~\ref{lem:fixJohn}, Lemma~\ref{lem:fixSq} or Table~\ref{table} given that $\mG$ is described in Theorem~\ref{thm:main}~$(ii)(a)$, $(ii)(b)$ or~$(ii)(c)$ respectively.

Suppose now that $r\ge 2$. The automorphism group of $\vG$ either embeds into $\Sym(m)\wr \Sym(r)$, where $m=|V\vG_i|$ for any $\vG_i\in \mG$, or, by maximality of $\Sym(m)\wr \Sym(r)$, $\Aut(\vG) = \Sym(m^r)$. In the latter scenario, $\vG \in \{\bL_m,\bK_m,\bL_m\cup\bK_m\}$.
All these graphs can be written as a merged product graph where $r=1$ and $\mJ$ is a Hamming set. This goes against our assumption on $\vG$, thus we must suppose $\Aut(\vG) \ne \Sym(m^r)$.

As a consequence, we obtain that, for some almost simple group $K$ listed in Theorem~\ref{thm:GurMag}~$(a)-(e)$, and for some transitive group $H\le \Sym(r)$, $K \wr H \le \Aut(\vG)$. Note that, as $K \le \Aut(\vG)^\Delta_\Delta$, by \cite[Theorem]{LIEBECK1987365}, $\Aut(\vG)^\Delta_\Delta$ is either $K$ or it contains $\Alt(m)$. If the latter case occurs,
then $\Alt(m)^r\wr H \le \Aut(\vG)$. By Lemma~\ref{lem:Ham}, $\vG$ is a generalised Hamming graph, which contradicts our choice of $\vG$. Therefore, $\Aut(\vG) \le K \wr \Sym(r)$.

Observe that we can apply Lemma~\ref{lem:relFixHam}. We obtain that
\[\RelFix (\vG) = 1 - \frac{\mu(K)m^{r-1}}{m^r} = 1 - \frac{\mu(K)}{m} = \RelFix\left(\MP(1,\mG,\mJ')\right) \,,\]
for some non-Hamming homogeneous set $\mJ'$. In particular, the relative fixities for $r\ge 2$ coincides with those we have already computed for $r=1$.
This complete the proof.
\end{proof}

\section{Proof of Theorem~\ref{thm:mainX}}
\label{sec:X}
Recall that a permutation group $G$ on $\Omega$ is \emph{quasiprimitive} if all its normal subgroups are transitive on $\Omega$. Clearly, any primitive group is quasiprimitive. Moreover, recall that, by repeating the proof of Cauchy--Frobenius~Lemma (see~\cite[Theorem~1.7A]{DM}) on the conjugacy class of a permutation $x\in G$, we get
\[\fix(x) |x^G| = |x^G \cap G_\omega|\]
where $\fix(x) = |\Omega| - |\supp(x)|$ is the number of fixed points of $x$.

\begin{proof}[Proof of Theorem~\ref{thm:mainX}]
(We would like to thank P.~Spiga again for pointing out the key ingredients for this proof.)
	Let $G$ be a quasiprimitive permutation group on a set $\Omega$,
	and let $x\in G\setminus\{1\}$ be an element achieving $|\supp(x)|\le(1-\alpha)|\Omega|$.
	For any point $\omega\in \Omega$, we obtain
	\[\alpha \le \frac{|x^G \cap G_\omega|}{|x^G|} \le \frac{|G_\omega|}{|x^G|} \le \frac{\beta}{|x^G|}\,.\]
        It follows that $|x^G|\le \alpha^{-1}\beta$.
	Now consider the normal subgroup of $G$ defined by
	\[N:=\bigcap\limits_{g\in G} \C_G(x^g)\,.\]
	Recall that $|G:\C_G(x)|=|x^G|$.
	Observe that $G$ acts by conjugation on the set
	\[\{\C_G(x^g) \mid g\in G\} \,,\]
	it defines a single orbit of  size $|x^G|$, and $N$ is the kernel of this action.
	Therefore
	\[|G:N| \le |x^G|! \le \left\lceil \frac{\beta}{\alpha} \right\rceil ! \,,\]
	that is, $N$ is a bounded index subgroup of $G$.
	Since $G$ is quasiprimitive, either $N$ is trivial or $N$ is transitive.
	Aiming for a contradiction, we suppose that $N$ is transitive.
	Since $[N,x]=1$, for any $\omega\in \Omega$ and for any $n\in N$,
	\[\omega^{nx}=\omega^{xn}=\omega^n \,.\]
	The transitivity of $N$ implies that $x = 1$, against our choice of $x$.
	Therefore, $N$ is trivial. It follows that
	\[|G| = |G:N| \le \left\lceil \frac{\beta}{\alpha} \right\rceil!\,.\]
	Since there are finitely many abstract groups of bounded size, the proof is complete.
\end{proof}

An equivalent formulation of Sims' Conjecture states that if $G$ is a primitive permutation group and the minimal out-valency among its nondiagonal orbital digraphs is at most $d$, then the size of a point stabilizer is bounded from above by a function $\f(d)$ depending only on the positive integer $d$.  An answer in the positive to this conjecture was given in \cite{CameronPraegerSaxlSeitz1983}.

\begin{proof}[Proof of Corollary~\ref{cor:graphs1}]
	Let $\vG$ be a vertex-primitive digraphs of out-valency at most $d$ and relative fixity exceeding $\alpha$, and let $G=\Aut(\vG)$. The hypothesis on the out-valency implies that, for any $v\in V\vG$, $|G_v|\le \f(d)$, where $\f(d)$ is the function that solves Sims' Conjecture. The result thus follows by choosing $\beta=\f(d)$ in Theorem~\ref{thm:mainX}.
\end{proof}

We conclude the paper by observing that, as $\f(d)\ge (d-1)!$, from Corollary~\ref{cor:graphs1} we cannot obtain a bound as sharp as that in Remark~\ref{cor:growth}.

\bibliographystyle{plain}
\bibliography{bibMinDeg}

\begin{thebibliography}{10}

\bibitem{Babai4}
L.~Babai.
\newblock On the order of uniprimitive permutation groups.
\newblock {\em Annals of Mathematics}, 113(3):553--568, 1981.

\bibitem{Babai1}
L.~Babai.
\newblock On the automorphism groups of strongly regular graphs {I}.
\newblock In {\em Proceedings of the 5th Conference on Innovations in
  Theoretical Computer Science}, ITCS '14, pages 359--368, New York, NY, USA,
  2014. Association for Computing Machinery.

\bibitem{Babai3}
L.~Babai.
\newblock Graph isomorphism in quasipolynomial time.
\newblock 2015.
\newblock \texttt{arXiv:1512.03547}.

\bibitem{Babai2}
L.~Babai.
\newblock On the automorphism groups of strongly regular graphs {II}.
\newblock {\em Journal of Algebra}, 421:560--578, 2015.
\newblock Special issue in memory of Ákos Seress.

\bibitem{BarbieriGrazianSpiga}
M.~Barbieri, V.~Grazian, and P.~Spiga.
\newblock On the number of fixed edges of automorphisms of vertex-transitive
  graphs of small valency.
\newblock {\em Journal of Algebraic Combinatorics}, 57:329--348, 2023.

\bibitem{MR1484478}
W.~Bosma, J.~Cannon, and C.~Playoust.
\newblock The {M}agma algebra system. {I}. {T}he user language.
\newblock {\em Journal of Symbolic Computation}, 24(3-4):235--265, 1997.
\newblock Computational algebra and number theory (London, 1993).

\bibitem{SRG}
A.~E. Brouwer and H.~Van Maldeghem.
\newblock {\em Strongly regular graphs}, volume 182 of {\em Encyclopedia of
  Mathematics and its Applications}.
\newblock Cambridge University Press, 2022.

\bibitem{BurnessGuralnick}
T.~C. Burness and R.~M. Guralnick.
\newblock Fixed point ratios for finite primitive groups and applications.
\newblock {\em Advances in Mathematics}, 411:108778, 2022.

\bibitem{Cameron2tr}
P.~J. Cameron.
\newblock Finite permutation groups and finite simple groups.
\newblock {\em Bulletin of the London Mathematical Society}, 13(1):1--22, 1981.

\bibitem{CameronPraegerSaxlSeitz1983}
P.~J. Cameron, C.~E. Praeger, J.~Saxl, and G.~M. Seitz.
\newblock On the {S}ims {C}onjecture and distance transitive graphs.
\newblock {\em Bulletin of the London Mathematical Society}, 15(5):499--506,
  1983.

\bibitem{CouttsQuickRoney-Dougal}
H.~J. Coutts, M.~Quick, and C.~M. Roney-Dougal.
\newblock The primitive permutation groups of degree less than $4\,096$.
\newblock {\em Communications in Algebra}, 39(10):3526--3546, 2011.

\bibitem{DM}
J.~D. Dixon and B.~Mortimer.
\newblock {\em Permutation Groups}, volume 163 of {\em Graduate Texts in
  Mathematics}.
\newblock Springer, New York, 1996.

\bibitem{GuralnickMagaard}
R.~Guralnick and K.~Magaard.
\newblock On the minimal degree of a primitive permutation group.
\newblock {\em Journal of Algebra}, 207(1):127--145, 1998.

\bibitem{ImrichKlavzar}
W.~Imrich and S.~Klavžar.
\newblock {\em Product Graphs: Structure and Recognition}.
\newblock Wiley-Interscience series in discrete mathematics and optimization.
  Wiley, 2000.

\bibitem{Jones2005}
G.~A. Jones.
\newblock Automorphisms and regular embeddings of merged johnson graphs.
\newblock {\em European Journal of Combinatorics}, 26(3):417--435, 2005.
\newblock Topological Graph Theory and Graph Minors, second issue.

\bibitem{JonesJaycay}
G.~A. Jones and R.~Jaycay.
\newblock Cayley properties of merged {J}ohnson graphs.
\newblock {\em Journal of Algebraic Combinatorics}, 44:1047--1067, 2016.

\bibitem{KantorLieber}
W.~M. Kantor and R.~A. Liebler.
\newblock The rank $3$ permutation representation of the finite classical
  groups.
\newblock {\em Transactions of the American Mathematical Society},
  271(1):1--71, 1982.

\bibitem{Kovacs}
L.~G. Kovács.
\newblock Wreath decompositions of finite permutation groups.
\newblock {\em Bulletin of the Australian Mathematical Society},
  40(2):255--279, 1989.

\bibitem{LehnerPotocnikSpiga}
F.~Lehner, P.~Potočnik, and P.~Spiga.
\newblock On fixity of arc-transitive graphs.
\newblock {\em Science China Mathematics}, 64:2603--2610, 2021.

\bibitem{Liebeck}
M.~W. Liebeck.
\newblock On minimal degrees and base sizes of primitive permutation groups.
\newblock {\em Archiv der Mathematik}, 43:11--15, 1984.

\bibitem{LIEBECK1987365}
M.~W. Liebeck, C.~E. Praeger, and J.~Saxl.
\newblock A classification of the maximal subgroups of the finite alternating
  and symmetric groups.
\newblock {\em Journal of Algebra}, 111(2):365--383, 1987.

\bibitem{LiebeckPraegerSaxl1988}
M.~W. Liebeck, C.~E. Praeger, and J.~Saxl.
\newblock On the {O}'{N}an--{S}cott theorem for finite primitive permutation
  groups.
\newblock {\em Journal of the Australian Mathematical Society}, 44(3):389--396,
  1988.

\bibitem{LiebeckSaxl}
M.~W. Liebeck and J.~Saxl.
\newblock Minimal degrees of primitive permutation groups, with an application
  to monodromy groups of covers of riemann surfaces.
\newblock {\em Proceedings of the London Mathematical Society},
  s3-63(2):266--314, 1991.

\bibitem{LiebeckShalev}
M.~W. Liebeck and A.~Shalev.
\newblock Simple groups, permutation groups, and probability.
\newblock {\em Journal of the American Mathematical Society}, 12:497--520,
  1999.

\bibitem{PotocnikSpiga}
P.~Potočnik and P.~Spiga.
\newblock On the number of fixed points of automorphisms of vertex-transitive
  graphs.
\newblock {\em Combinatorica}, 41:703--747, 2021.

\bibitem{PraegerSchneider}
C.~E. Praeger and C.~Schneider.
\newblock Embedding permutation groups into wreath products in product action.
\newblock {\em Journal of the Australian Mathematical Society}, 92(1):127--136,
  2012.

\end{thebibliography}
\end{document}